\documentclass[a4paper,12pt]{article}
\usepackage{amsmath,amssymb,graphicx}
\graphicspath{{./fig/}}
\usepackage{amsfonts, amsmath}
\usepackage{graphicx}
\usepackage{subcaption}
\usepackage{rotating}
\usepackage{psfrag}

\usepackage[utf8]{inputenc}
\usepackage[english]{babel}
\usepackage{amsmath}
\usepackage{amsthm}
\usepackage{amssymb}
\usepackage{a4wide}
\usepackage{bm}
\usepackage{microtype}
\usepackage{mathtools}
\usepackage{csquotes}

\usepackage{enumitem}
\setlist[enumerate]{itemsep=0mm,parsep=2mm}

\usepackage[dvipsnames]{xcolor}
\usepackage{hyperref}
\hypersetup{
 linkcolor = black,
 citecolor = black,
 urlcolor = Aquamarine!85!black,
 colorlinks = true,
}
\usepackage{tikz, graphicx, standalone, caption, subcaption, wrapfig}
\usetikzlibrary{calc,decorations.text}

\usepackage{subcaption}

\newtheorem{thm}{Theorem}
\newtheorem{lem}[thm]{Lemma}

\newtheorem{defn}[thm]{Definition}

\newtheorem{conj}[thm]{Conjecture}

\newtheorem{question}[thm]{Question}

\newcommand{\R}{\mathbb{R}}

\begin{document}

\title{On generic universal rigidity on the line} 
\author{Guilherme Zeus Dantas e Moura\thanks{Department of Mathematics and Statistics, 
 Haverford College, 370 Lancaster Ave,
 Haverford, PA 19041, USA.
 e-mail: \texttt{zeusdanmou@gmail.com}
 }\and
 Tibor Jord\'an\thanks{
 Department of Operations Research,
 ELTE E\"otv\"os Lor\'and University, and the 
 ELKH-ELTE Egerv\'ary Research Group on Combinatorial Optimization, E\"otv\"os Lor\'and Research Network (ELKH),
 P\'azm\'any P\'eter s\'et\'any 1/C,
 1117 Budapest, Hungary.
 e-mail: \texttt{tibor.jordan@ttk.elte.hu}
 }\and
 Corwin Silverman\thanks{
 Department of Mathematics,
 Grinnell College, 1115 8th Ave,
 Grinnell, IA 50112, USA.
 email: \texttt{silvermanc79@gmail.com}
 }
}

\maketitle


\begin{abstract}
A $d$-dimensional bar-and-joint framework $(G,p)$
with underlying graph $G$
is called universally
rigid if all realizations of $G$ with the same edge lengths, in all
dimensions, are congruent to $(G,p)$. A graph $G$ is said to be
generically universally rigid in $\R^d$ if every $d$-dimensional generic framework $(G,p)$ is universally
rigid.

In this paper we focus on the case $d=1$. We give counterexamples to a
conjectured characterization of generically universally rigid graphs from \cite{conwshop}.
We also introduce two new operations that preserve the universal rigidity of
generic frameworks, and the property of being not universally rigid,
respectively. One of these operations is used in the analysis of one of our
examples, while the other operation is applied to obtain a lower bound on
the size of generically universally rigid graphs. This bound gives
a partial answer to a question from \cite{JN}. 

\vspace{1em}
\end{abstract}

\section{Introduction}

A $d$-dimensional (bar-and-joint) {\it framework} is a pair $(G,p)$, 
where $G=(V,E)$
is a graph and $p$ is a {\it configuration} of the vertices, that is, a
map from $V$ to $\mathbb{R}^d$.
We consider the framework to be a straight line {\it realization} of $G$ in
$\mathbb{R}^d$.
Two frameworks
$(G,p)$ and $(G,q)$ are {\it equivalent} if
$\|p(u)-p(v)\|=\|q(u)-q(v)\|$ holds
for all pairs $u,v$ with $uv\in E$,
where $\|.\|$ denotes the Euclidean norm in $\mathbb{R}^d$.
Frameworks $(G,p)$, $(G,q)$ are {\it congruent} if
$\|p(u)-p(v)\|=\|q(u)-q(v)\|$ holds
for all pairs $u,v$
with $u,v\in V$.
This is the same as saying that $(G,q)$ can be
obtained from $(G,p)$ by
an isometry
of $\mathbb{R}^d$.

Let $(G,p)$ be a $d$-dimensional framework for some $d\geq 1$.
We say that $(G,p)$
is {\it rigid} in $\mathbb{R}^d$ if 
there is a neighborhood $U_{p}$ in the space of $d$-dimensional configurations 
such that if a $d$-dimensional framework
$(G,q)$ is equivalent to $(G,p)$ and $q \in U_{p}$, then 
$q$ is congruent to $p$. The framework $(G,p)$ is called {\it globally rigid}
in $\mathbb{R}^d$ if every $d$-dimensional framework $(G,q)$ which is equivalent
to $(G,p)$ is congruent to $(G,p)$.
We obtain an even stronger property by extending this condition to
equivalent realizations in any dimension: we say that
$(G,p)$ is {\it universally rigid}
if it is a unique realization of $G$, up to congruence, with the given edge
lengths, in all dimensions $\mathbb{R}^{d'}$, $d'\geq 1$. 

Deciding whether a given framework is rigid in $\R^d$, for $d\geq 2$
(resp. globally rigid in $\R^d$, for $d\geq 1$) is NP-hard 
\cite{Abb, Saxe}. The complexity of the corresponding decision problem for
universal rigidity seems to be open, even for $d=1$.
These problems become more tractable, however, if we assume that there are no
algebraic dependencies between the coordinates of the points of
the framework.
A framework $(G,p)$ is said to be {\it generic} if
the set containing the coordinates of all its points is
algebraically independent over the rationals.
It is
well-known 
that the rigidity (resp. global rigidity) of frameworks in $\mathbb{R}^d$
is a generic property for all $d\geq 1$, that is, the (global)
rigidity of $(G,p)$ depends only on the graph $G$ and not the particular
realization $p$,
if $(G,p)$ is generic \cite{AR,connelly2,GHT}.
This property does not hold for universal rigidity, even if
$d=1$, which follows by considering different generic
realizations of a four-cycle on the line. See Figure~\ref{fig:fourcycle}.

\begin{figure}[ht]
 \centering
 \begin{subfigure}[b]{0.35\textwidth}
 \centering
 \includegraphics{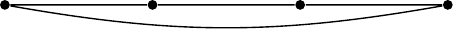}
 \caption{An universally rigid realization of a four-cycle in $\mathbb{R}^1$.}
 \label{fig:fourcycle:a}
 \end{subfigure}
 \hspace{0.1\textwidth}
 \begin{subfigure}[b]{0.35\textwidth}
 \centering
 \includegraphics{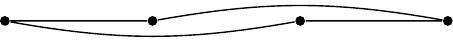}
 \caption{A not universally rigid realization of a four-cycle in $\mathbb{R}^1$.}
 \label{fig:fourcycle:b}
 \end{subfigure}\\[1em]
 \begin{subfigure}[b]{0.85\textwidth}
 \centering
 \includegraphics{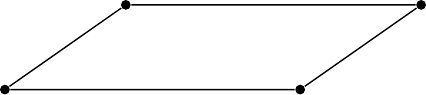}
 \caption{A realization of a four-cycle in $\mathbb{R}^2$ equivalent, but not
 congruent, to the realization in Figure~\ref{fig:fourcycle:b}.}
 \label{fig:fourcycle:c}
 \end{subfigure}
 \caption{Realizations of a four-cycle.}
 \label{fig:fourcycle}
\end{figure}

A graph $G$ is called
{\it generically rigid} (resp. {\it generically globally rigid}, {\it generically
universally rigid})
 in ${\mathbb R}^d$ if every $d$-dimensional generic framework $(G,p)$ is
rigid (resp. globally rigid, universally rigid).
Generically rigid and globally rigid graphs are well-characterized for $d\leq 2$.
It remains an open problem to extend these results to $d\geq 3$.
The characterization of generically universally rigid graphs is an open problem for all $d\geq 1$.
We refer the reader to \cite{JW,SW} for more details on the theory
of rigid graphs and frameworks.


In this paper we focus on universally rigid frameworks and
generically universally rigid graphs in $\R^1$.
We give counterexamples to a
conjectured characterization of generically universally rigid graphs from \cite{conwshop}.
We introduce a new operation that preserves the universal rigidity of
generic frameworks and use it to construct infinite families of counterexamples.
We also show that the so-called degree-2 extension operation
preserves the property of being not universally rigid, for $d=1$.
This operation is applied in the proof of a new lower bound on
the size of generically universally rigid graphs. This bound gives
a partial answer to a question from \cite{JN}.

\section{Examples}

Let $G_1=(V_1,E_1)$, $G_2=(V_2,E_2)$ be two graphs
for which $V_1\cap V_2$, $V_1-V_2$, and $V_2-V_1$ are all nonempty.
Then $G=(V_1\cup V_2, E_1 \cup (E_2-E(G_2[V_1\cap V_2])))$ is called
the {\it edge reduced attachment} of $G_1$ and $G_2$ {\it along $G_2$}.
That is, $G$ is obtained by
removing the edges of $G_2$ which are spanned by the intersection of
their vertex sets and then taking the union of the two graphs.
Ratmanski \cite{ratmanski} proved that the edge reduced attachment operation preserves
generic universal rigidity in $\R^d$, provided $|V_1\cap V_2|\geq d+1$.
It was conjectured, by participants of a workshop in 2011, that
for $d=1$ every generically universally rigid graph can be obtained from 
a set of triangles by this operation, and edge addition.

\begin{conj} \cite{conwshop}
\label{inductive} A graph $G$ on at least three vertices
is generically universally rigid in $\R^1$ if and only if $G$ can be obtained from
a set of triangles by edge reduced attachment and edge addition operations.
\end{conj}

We next present two counterexamples to Conjecture~\ref{inductive}.
The first one, on eight vertices, requires a more sophisticated argument,
including the analysis of a new operation that can be used to build
generically universally rigid graphs. 
This operation can also be used to construct infinite families of counterexamples.
The second one, on sixteen vertices,
is fairly easy to verify.
We need the following simple lemma on attachments.

The {\it $K_4$-completion} operation adds a new edge $uv$ to a graph
for a vertex pair $u,v$ with two adjacent common neighbours.
We say that $F\subseteq E$ is
an {\it independent edge cut} in a graph $G=(V,E)$ 
if the edges in $F$ are pairwise disjoint, and there is a nonempty proper
subset $X\subset V$ for which the set of edges connecting $X$ and $V-X$ is $G$
is $F$. 

\begin{lem}
\label{3prop}
Suppose that $G=(V,E)$ is a connected graph that can be obtained by edge reduced attachments and edge additions from a set of
triangles. 
Then\\
(i) $G$ contains a triangle,\\
(ii) the complete graph on $V$ can be obtained from $G$ by $K_4$-completion operations,\\
(iii) $G$ has no independent edge cuts.
\end{lem}

\begin{proof}
(i) follows, by induction, from the fact that if $G$ is the edge reduced attachment of $G_1$ and $G_2$ along $G_2$, then
$G$ contains $G_1$ as a subgraph. So if $G_1$ contains a triangle, so does $G$.
To prove (ii) we use induction on the number $t$ of operations used to build up $G$.
For $t=0$ we have $G=K_3$, for which the statement is obvious.
Suppose that $t\geq 1$ and consider the last operation applied, that resulted in graph $G$. 
The case when the last operation is edge addition is easy to deal with, so suppose that $G$
was obtained from $G_1$ and $G_2$ by an edge reduced attachment along $G_2$.
Then $G_1$ is a subgraph of $G$. By induction the complete graph on $V(G_1)$ can be obtained from $G_1$ by
$K_4$-completions. By performing these operations on $G$, we make $V(G_1)\cap V(G_2)$ complete, which implies
that $G_2$ is a subgraph of the resulting graph. So, by induction, we can apply $K_4$-completions to make
the subgraph of $G$ on $V(G_2)$ complete as well.
Since $G_1$ and $G_2$ share at least two vertices, all the missing edges of $G$ can then be added
by further $K_4$-completions. 
Finally, it is clear that (ii) implies (iii).
\end{proof}

\begin{figure}[ht]
 \centering
 \begin{subfigure}[b]{0.3\textwidth}
 \centering
 \includegraphics{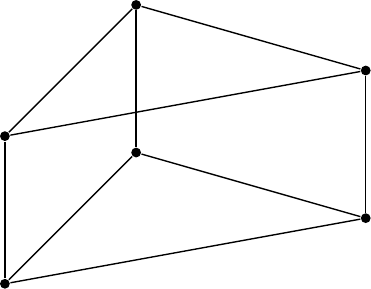}
 \caption{The graph $B_3$.}
 \label{fig:burger:3}
 \end{subfigure}
 \hspace{0.03\textwidth}
 \begin{subfigure}[b]{0.3\textwidth}
 \centering
 \includegraphics{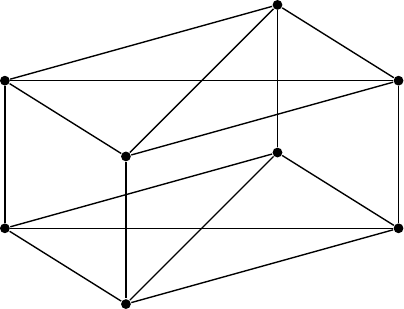}
 \caption{The graph $B_4$.}
 \label{fig:burger:4}
 \end{subfigure}
 \hspace{0.03\textwidth}
 \begin{subfigure}[b]{0.3\textwidth}
 \centering
 \includegraphics{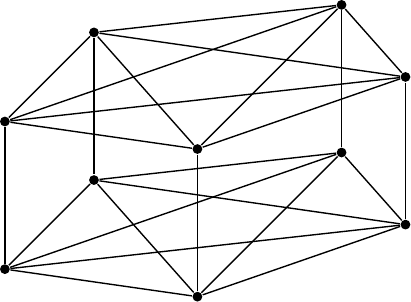}
 \caption{The graph $B_5$.}
 \label{fig:burger:5}
 \end{subfigure}
 \caption{Graphs $B_3$, $B_4$, and $B_5$.}
 \label{fig:burger}
\end{figure}

Let $B_n$ be the graph obtained from two disjoint complete graphs on $n$ vertices by adding
$n$ disjoint edges. See Figure~\ref{fig:burger}.
As we shall see in the next section (c.f. Theorem~\ref{join}), $B_n$ is generically universally rigid in $\R^1$ for $n\geq 4$.
Since $B_n$ has an independent edge cut, it cannot be obtained by edge reduced attachments and edge additions from a set of
triangles by Lemma~\ref{3prop}(iii). 
Hence $B_4$ (and each graph $B_i$, for $i\geq 4$) is a counterexample to Conjecture~\ref{inductive}.

The other example was motivated by a question in \cite{JN}, asking whether there is a
triangle-free generically universally rigid graph in $\R^1$, and in particular, whether
the triangle-free 
Gr\"otzsch graph is generically universally rigid in $\R^1$. See Figure~\ref{fig:Grotzsch}.
We leave the latter question open. Instead we consider the following supergraph of the
Gr\"otzsch graph and use it to give an affirmative answer to the former question.

\begin{defn}
The \textit{augmented Gr\"otzsch graph} 
is obtained from the Gr\"otzsch graph
on vertex set $\{w,u_0,u_1,\dots u_4,v_0,v_1\dots v_4\}$ (labeled as in Figure~\ref{fig:Grotzsch}) by adding the vertices
$v_i'$, $0\leq i\leq 4$ and edges $v_i'u_{i+1},v_i'v_{i+1},v_i'v_{i+1}'$, for $0\leq i\leq 4$,
counting indices mod $5$. See Figure~\ref{fig:augmentedGrotzsch}.
\end{defn}


\begin{figure}[ht]
 \centering
 \begin{subfigure}[c]{.4\textwidth}
 \centering
 \includegraphics{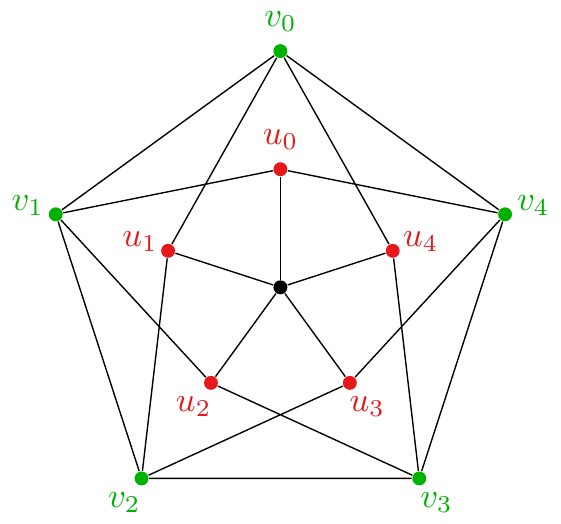}
 \caption{The Grötzsch graph.}
 \label{fig:Grotzsch}
 \end{subfigure}
 \hspace{.05\textwidth}
 \begin{subfigure}[c]{.5\textwidth}
 \centering
 \includegraphics{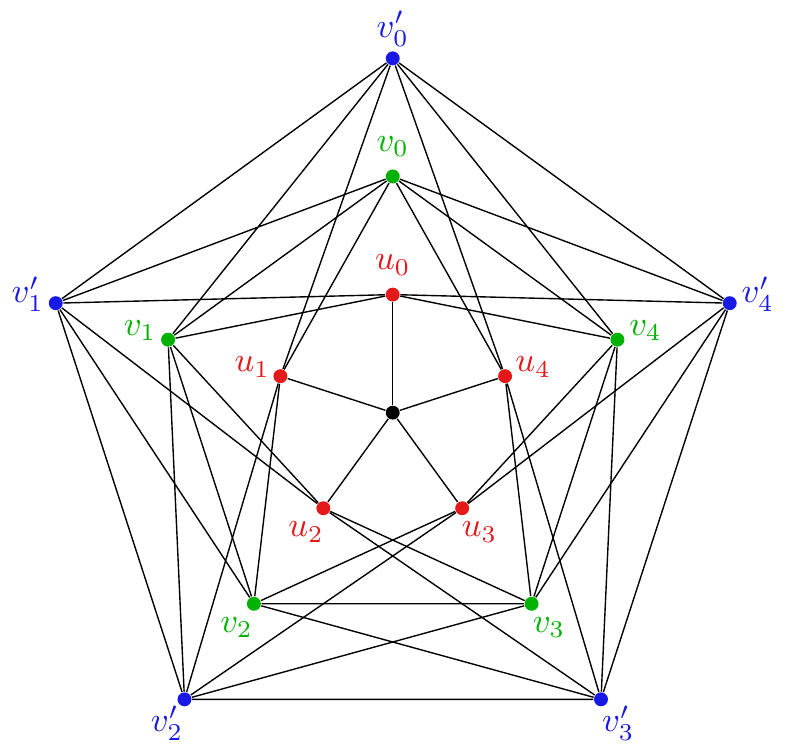}
 \caption{The augmented Grötzsch graph.}
 \label{fig:augmentedGrotzsch}
 \end{subfigure}
 \caption{The Grötzsch graph and the augmented Grötzsch graph. The central vertex is $w$.}
\end{figure}

Thus
the augmented Gr\"otzsch graph has sixteen vertices, and it contains the Gr\"otzsch graph as a subgraph.
We shall prove that it is generically universally rigid in $\R^1$.
We need the following simple lemma. 
The {\it degree-2 extension} operation adds a new vertex $w$ to a graph $G$ and two
new edges $wx, wy$, for two distinct vertices of $G$.
This operation can also be performed on a framework $(G,p)$, in which case
it includes the extension of $p$ by $p(w)$. 

\begin{lem}
\label{2ext}
Let $(G,p)$ be a universally rigid realization of $G$ in $\R^1$. 
Suppose that $(G',p)$ can be obtained from $(G,p)$ by a degree-2 extension that adds
the edges $wx, wy$. If $p(x)\not= p(y)$,
then $(G',p)$ is universally rigid in $\R^1$.
\end{lem}

Let $C$ be a cycle on vertex set $\{x_1,x_2,\ldots, x_k\}$ 
with edge set $E(C)=\{x_1x_2,\ldots, x_{k-1}x_k, \allowbreak x_kx_1\}$. 
Let $(C,p)$ be a $1$-dimensional realization of $C$.
If $p(x_1)<p(x_2)<\cdots<p(x_k)$ then $(C,p)$ (or sometimes $C$ itself) is called a \emph{stretched cycle}. 
It is easy to see that stretched cycles are universally rigid in $\R^1$.

It is known that the Gr\"otzsch graph is not a ``cover graph.''
By using our terminology, this fact can be restated as follows.
See \cite{FFLW} for a short combinatorial proof.

\begin{thm} \cite{FFLW}
\label{notcover}
Every injective $1$-dimensional realization of the Gr\"otzsch graph has a stretched cycle.
\end{thm}


We are ready to deduce that the augmented Gr\"otzsch graph is generically universally rigid.
We can actually show a somewhat stronger property.

\begin{thm}
\label{mainG}
Every injective $1$-dimensional realization of the augmented Gr\"otzsch graph is universally rigid.
\end{thm}

\begin{proof}
Let $G$ be the augmented Gr\"otzsch graph and let $G'$ denote its subgraph
isomorphic to the Gr\"otzsch graph, obtained by deleting the vertices $v_i'$, $0\leq i\leq 4$.
Consider an injective $1$-dimensional realization $(G,p)$.
By Theorem~\ref{notcover} there is a stretched cycle $C$ on vertices $\{x_1,x_2,\dots, x_k\}$ in $(G',p|_{V(G')})$.
Since $G'$ is triangle-free, we must have $k\geq 4$.
Let $(H,p|_{V(H)})$ be a maximal universally rigid subframework of $(G,p)$ with $V(C)\subseteq V(H)$.
Such a framework exists, since the subframework of the stretched cycle $C$ is universally rigid. 
We shall prove that $H=G$.


The structure of $G'$ and the fact that $C$ contains a path on three
vertices which is disjoint from $w$ implies that we must have two vertices in $C-w$ with identical
indices or two vertices whose indices differ by two. Formally, 
either (1) there exists a pair of vertices $u_i,v_i$ in $C$,
or (2) there exists a pair of vertices $u_i,v_{i+2}$ (or $u_i,u_{i+2}$, $v_i,v_{i+2}$, $v_i,u_{i+2}$),
for some $0\leq i\leq 4$.

Let us consider case (1). By symmetry we may assume that $u_1,v_1\in V(C)$. 
Then, since $u_1,v_1$ are both neighbours of $v_2$ and $v_2'$,
Lemma~\ref{2ext} implies that $v_2,v_2'\in V(H)$ holds. We can apply a similar argument
to $v_3,v_3'$, and so on around the cycle, to deduce that $H$ contains all $v$- and $v'$-vertices.
Finally, the $u$-vertices and vertex $w$ can also be added by degree-2 extensions.
Thus $H=G$ and the theorem follows.

Next consider case (2). Suppose that, say, $u_1,v_3\in V(C)$. Then, since 
$u_1,v_3$ are both neighbours of $v_2$ and $v_2'$,
Lemma~\ref{2ext} implies that $v_2,v_2'\in V(H)$ holds.
The rest of the argument is identical to that of case (1).
This completes the proof.
\end{proof}

Since the augmented Gr\"otzsch graph is triangle-free and, by Theorem~\ref{mainG}, generically
universally rigid in $\R^1$, it follows from Lemma~\ref{3prop}(i) that it is a
counterexample to Conjecture~\ref{inductive}.

It may be interesting to find 
graphs with arbitrarily large girth which are generically
universally rigid in $\R^1$ (or possibly in higher dimensions).

\section{Operations}

\subsection{Combining graphs along disjoint edges}

If $G$ has an independent edge cut $J$, then the removal of $J$ from $G$
results in two smaller (sub)graphs. The reversal of this operations can be
defined as follows. Let \(G = (V, E)\), \(H = (U, F)\) be two disjoint graphs.
Let \(v_1, v_2, \dots, v_k \in V\) and \(u_1, u_2, \dots, u_k \in U\) be distinct vertices.
The {\it join} $G \sqcup H$ of \(G\) and \(H\) is the graph
on vertex set $V\cup U$ with edge set $E\cup F\cup
%
\{ v_iu_i : i \in \{1, 2, \dots, k\} \}.$
We say that $G \sqcup H$ is obtained from $G$ and $H$ by a join operation {\it along} $k$ {\it edges}.
See Figure~\ref{fig:join}.
For example, $B_n$ is the join of two complete graphs on $n$ vertices along $n$ edges.

\begin{figure}[ht]
 \centering
 \includegraphics{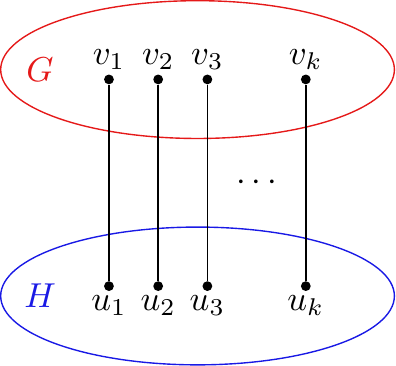}
 \caption{The graph obtained from $G$ and $H$ by a join operation along $k$ edges.}
 \label{fig:join}
\end{figure}

\begin{thm}
\label{join}
Let $G$ and $H$ be generically universally rigid graphs in $\R^1$ on at least $k$ vertices and
let $G \sqcup H$ be obtained from $G$ and $H$ by a join operation along $k$ edges.
If $k\geq 4$, then $G \sqcup H$ is generically universally rigid in $\R^1$.
\end{thm}

\begin{proof}
Let \(G = (V, E)\), \(H = (U, F)\),
$B=G \sqcup H$. We may assume that $k=4$.
Consider a $1$-dimensional generic realization $(B,p)$.
Suppose that $(B,q)$ is an equivalent realization in $\R^d$, for some $d\geq 1$.
	We shall prove that \(p\) and \(q\) are congruent.

	Note that 
	$(G,p|_V)$ is a generic \(1\)-dimensional realization of \(G\), and
	$(G,q|_V)$ is an equivalent realization. Hence, since \(G\) is generically universally rigid in $\R^1$, 
	it follows that 
	$(G,p|_V)$ is congruent to $(G,q|_V)$.
	Analogously, we also have that $(H,p|_U)$ is congruent to $(H,q|_U)$.
	%
	Hence, there exists \(\mathbf{c}_G, \mathbf{d}_G, \mathbf{c}_H, \mathbf{d}_H \in \mathbb{R}^d\), with \(\|\mathbf{d}_G\| = \|\mathbf{d}_H\| = 1\), such that 
	\begin{align*}
		q(v) = \mathbf{c}_G + p(v) \mathbf{d}_G &\qquad \forall v \in V, \qquad\text{and} \\
		q(u) = \mathbf{c}_H + p(u) \mathbf{d}_H &\qquad \forall u \in U.
	\end{align*}


	\begin{figure}[ht]
		\centering
 \includegraphics{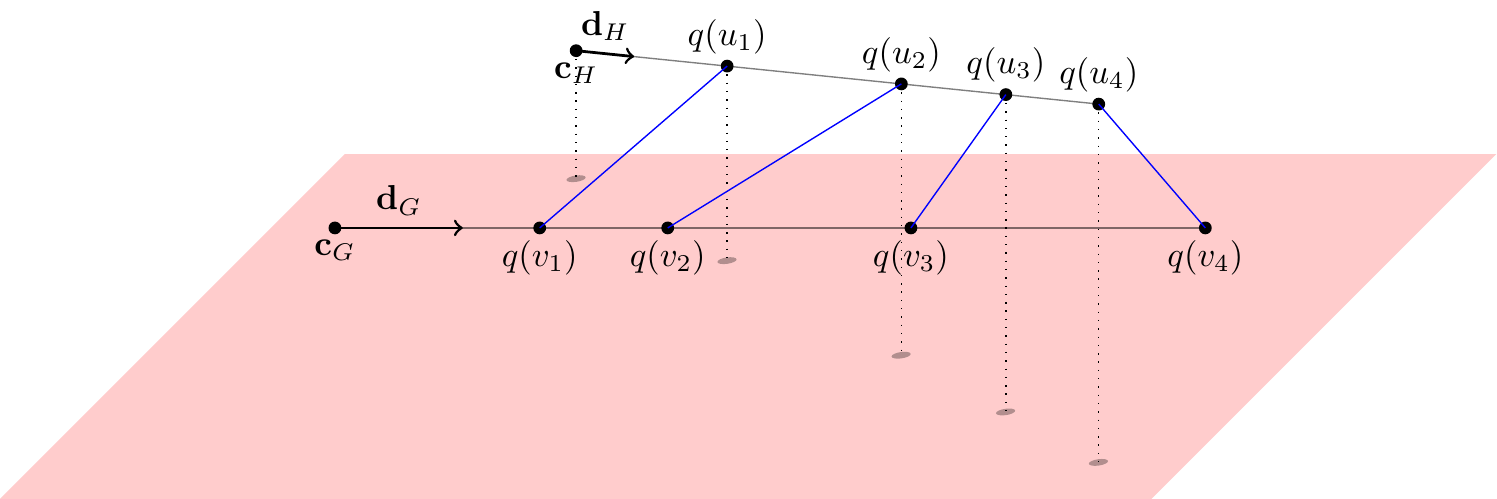}
		\caption{}
	\end{figure}


\noindent	Since \(p\) and \(q\) are equivalent, for each \(i \in \{1, 2, 3, 4\}\), we have
	\begin{equation*}
		\| q(v_i) - q(u_i) \|^2 = \| p(v_i) - p(u_i) \|^2.
	\end{equation*}

\noindent	By changing the square of norms to dot product we obtain that, for each \(i \in \{1, 2, 3, 4\}\),
	\begin{equation*}
		\Big( (\mathbf{c}_G + p(v_i)\mathbf{d}_G) - (\mathbf{c}_H + p(u_i)\mathbf{d}_H) \Big)
		\cdot
		\Big( (\mathbf{c}_G + p(v_i)\mathbf{d}_G) - (\mathbf{c}_H + p(u_i)\mathbf{d}_H) \Big)
		- (p(v_i) - p(u_i))^2 = 0.
	\end{equation*}

\noindent	This gives, by using \(\|\mathbf{d}_G\| = \|\mathbf{d}_H\| = 1\), that for each \(i \in \{1, 2, 3, 4\}\),
 \begin{equation} \label{eq:four-terms}
		\begin{aligned}
			\mathbf{c}_G \cdot \mathbf{c}_G - 2 \mathbf{c}_G \cdot \mathbf{c}_H + \mathbf{c}_H \cdot \mathbf{c}_H + {} & \\
			2(1 - \mathbf{d}_G \cdot \mathbf{d}_H) \, p(u_i) p(v_i) + {} & \\
			2(\mathbf{c}_H \cdot \mathbf{d}_H - \mathbf{c}_G \cdot \mathbf{d}_H ) p(u_i) + {} & \\
			2(\mathbf{c}_G \cdot \mathbf{d}_G - \mathbf{c}_H \cdot \mathbf{d}_G) p(v_i) &= 0.
		\end{aligned}
	\end{equation}

\noindent By applying equation~\eqref{eq:four-terms} for \(i \in \{1, 2, 3\}\), as well as for \(i = 4\), and by subtracting, we obtain that for each \(i \in \{1, 2, 3\}\),
	\begin{equation*}
		\begin{aligned}
			{ 2(1 - \mathbf{d}_G \cdot \mathbf{d}_H) \, (p(u_i) p(v_i) - p(u_4) p(v_4))} & \mathrel{+} \\
			{ 2(\mathbf{c}_H \cdot \mathbf{d}_H - \mathbf{c}_G \cdot \mathbf{d}_H ) (p(u_i) - p(u_4)) } & \mathrel{+} & \\
			{2(\mathbf{c}_G \cdot \mathbf{d}_G - \mathbf{c}_H \cdot \mathbf{d}_G) (p(v_i) - p(v_4)) } &= 0.
		\end{aligned}
	\end{equation*}

\noindent Let us	define \({f_1(i) = p(u_i) p(v_i) - p(u_4) p(v_4)}\), \({f_2(i) = p(u_i) - p(u_4)}\), and \({f_3(i) = p(v_i) - p(v_4)}\).
 Then we have, for each \(i \in \{1, 2, 3\}\),
	\begin{equation} \label{eq:three-terms}
		{2(1 - \mathbf{d}_G \cdot \mathbf{d}_H) f_1(i)} + 
		{2(\mathbf{c}_H \cdot \mathbf{d}_H - \mathbf{c}_G \cdot \mathbf{d}_H) f_2(i)} + 
		{2(\mathbf{c}_G \cdot \mathbf{d}_G - \mathbf{c}_H \cdot \mathbf{d}_G) f_3(i)} = 0.
	\end{equation}

\noindent
By applying equation~\eqref{eq:three-terms} for \(i \in \{1, 2\}\), as well as for \(i = 3\), and by multiplying the respective equations by \(f_3(3)\) and \(f_3(i)\), and then subtracting, we obtain that for each \(i \in \{1, 2\}\),
	\begin{equation*}
		{2(1 - \mathbf{d}_G \cdot \mathbf{d}_H)} ( f_1(i)f_3(3) - f_1(3)f_3(i) ) + 
		{2(\mathbf{c}_H \cdot \mathbf{d}_H - \mathbf{c}_G \cdot \mathbf{d}_H)} ( f_2(i)f_3(3) - f_2(3)f_3(i)) = 0
	\end{equation*}

\noindent
Let us	define \({f_4(i) = f_1(i)f_3(3) - f_1(3)f_3(i)}\), and \({f_5(i) = f_2(i)f_3(3) - f_2(3)f_3(i)}\).
	Then we have
	\begin{gather}
		{2(1 - \mathbf{d}_G \cdot \mathbf{d}_H) f_4(1)} + 
		{2(\mathbf{c}_H \cdot \mathbf{d}_H - \mathbf{c}_G \cdot \mathbf{d}_H) f_5(1)} = 0, \label{eq:two-terms1} \\
		{2(1 - \mathbf{d}_G \cdot \mathbf{d}_H) f_4(2)} + 
		{2(\mathbf{c}_H \cdot \mathbf{d}_H - \mathbf{c}_G \cdot \mathbf{d}_H) f_5(2)} = 0. \label{eq:two-terms2}
	\end{gather}

\noindent
Let us	define \(f = f_4(1)f_5(2) - f_4(2)f_5(1)\). By multiplying equation~\eqref{eq:two-terms1} by \(f_5(2)\), equation~\eqref{eq:two-terms2} by \(f_5(1)\), and taking the difference, we obtain
	\begin{equation*}
		2(1 - \mathbf{d}_G \cdot \mathbf{d}_H) f = 0.
	\end{equation*} 

	Therefore, \(f = 0\) or \(\mathbf{d}_G \cdot \mathbf{d}_H = 1\) must hold.
In the former case, \(f\) is a non-zero rational polynomial in eight variables with \(p(v_1)\), \(p(v_2)\), \(p(v_3)\), \(p(v_4)\), \(p(u_1)\), \(p(u_2)\), \(p(u_3)\), and \(p(u_4)\) as a root, which contradicts the genericity of \(p\).

The explicit form of the polynomial is as follows:
\begin{equation*}
  \begin{aligned}
    f =
    -& {p(u_1)} {p(u_3)} {p(v_1)} {p(v_2)} {p(v_3)} 
    + {p(u_2)} {p(u_3)} {p(v_1)} {p(v_2)} {p(v_3)} 
    + {p(u_1)} {p(u_4)} {p(v_1)} {p(v_2)} {p(v_3)} \\
    -& {p(u_2)} {p(u_4)} {p(v_1)} {p(v_2)} {p(v_3)} 
    + {p(u_1)} {p(u_2)} {p(v_1)} {p(v_3)}^{2} 
    - {p(u_2)} {p(u_3)} {p(v_1)} {p(v_3)}^{2} \\
    -& {p(u_1)} {p(u_4)} {p(v_1)} {p(v_3)}^{2} 
    + {p(u_3)} {p(u_4)} {p(v_1)} {p(v_3)}^{2} 
    - {p(u_1)} {p(u_2)} {p(v_2)} {p(v_3)}^{2} \\
    +& {p(u_1)} {p(u_3)} {p(v_2)} {p(v_3)}^{2} 
    + {p(u_2)} {p(u_4)} {p(v_2)} {p(v_3)}^{2} 
    - {p(u_3)} {p(u_4)} {p(v_2)} {p(v_3)}^{2} \\
    +& {p(u_1)} {p(u_3)} {p(v_1)} {p(v_2)} {p(v_4)} 
    - {p(u_2)} {p(u_3)} {p(v_1)} {p(v_2)} {p(v_4)} 
    - {p(u_1)} {p(u_4)} {p(v_1)} {p(v_2)} {p(v_4)} \\
    +& {p(u_2)} {p(u_4)} {p(v_1)} {p(v_2)} {p(v_4)} 
    - 2 {p(u_1)} {p(u_2)} {p(v_1)} {p(v_3)} {p(v_4)} 
    + {p(u_1)} {p(u_3)} {p(v_1)} {p(v_3)} {p(v_4)} \\
    +& {p(u_2)} {p(u_3)} {p(v_1)} {p(v_3)} {p(v_4)} 
    + {p(u_1)} {p(u_4)} {p(v_1)} {p(v_3)} {p(v_4)} 
    + {p(u_2)} {p(u_4)} {p(v_1)} {p(v_3)} {p(v_4)} \\
    -& 2 {p(u_3)} {p(u_4)} {p(v_1)} {p(v_3)} {p(v_4)} 
    + 2 {p(u_1)} {p(u_2)} {p(v_2)} {p(v_3)} {p(v_4)} 
    - {p(u_1)} {p(u_3)} {p(v_2)} {p(v_3)} {p(v_4)} \\
    -& {p(u_2)} {p(u_3)} {p(v_2)} {p(v_3)} {p(v_4)} 
    - {p(u_1)} {p(u_4)} {p(v_2)} {p(v_3)} {p(v_4)} 
    - {p(u_2)} {p(u_4)} {p(v_2)} {p(v_3)} {p(v_4)} \\
    +& 2 {p(u_3)} {p(u_4)} {p(v_2)} {p(v_3)} {p(v_4)} 
    - {p(u_1)} {p(u_3)} {p(v_3)}^{2} {p(v_4)} 
    + {p(u_2)} {p(u_3)} {p(v_3)}^{2} {p(v_4)} \\
    +& {p(u_1)} {p(u_4)} {p(v_3)}^{2} {p(v_4)} 
    - {p(u_2)} {p(u_4)} {p(v_3)}^{2} {p(v_4)} 
    + {p(u_1)} {p(u_2)} {p(v_1)} {p(v_4)}^{2} \\
    -& {p(u_1)} {p(u_3)} {p(v_1)} {p(v_4)}^{2} 
    - {p(u_2)} {p(u_4)} {p(v_1)} {p(v_4)}^{2} 
    + {p(u_3)} {p(u_4)} {p(v_1)} {p(v_4)}^{2} \\
    -& {p(u_1)} {p(u_2)} {p(v_2)} {p(v_4)}^{2} 
    + {p(u_2)} {p(u_3)} {p(v_2)} {p(v_4)}^{2} 
    + {p(u_1)} {p(u_4)} {p(v_2)} {p(v_4)}^{2} \\
    -& {p(u_3)} {p(u_4)} {p(v_2)} {p(v_4)}^{2} 
    + {p(u_1)} {p(u_3)} {p(v_3)} {p(v_4)}^{2} 
    - {p(u_2)} {p(u_3)} {p(v_3)} {p(v_4)}^{2} \\
    -& {p(u_1)} {p(u_4)} {p(v_3)} {p(v_4)}^{2} 
    + {p(u_2)} {p(u_4)} {p(v_3)} {p(v_4)}^{2}.
  \end{aligned}
\end{equation*} 

In the latter case, \(\mathbf{d}_G \cdot \mathbf{d}_H = 1\) (along with \(\|\mathbf{d}_G\| = \|\mathbf{d}_H\| = 1\)) imply that \(\mathbf{d}_G = \mathbf{d}_H\).
 Let \(\mathbf{d} = \mathbf{d}_G = \mathbf{d}_H\)	and \(\mathbf{t} = \mathbf{c}_H - \mathbf{c}_G\).

Suppose that \(\mathbf{t} \neq \mathbf{0}\). By applying an isometry of $\R^d$ to $(B,q)$ we may suppose that 
$p|_V=q|_V$. This implies that $(H,p|_U)$ and $(H,q|_U)$ are two congruent realizations of $H$
on two parallel lines $L_p$ and $L_q$, with the same vertex ordering.
Furthermore, $L_p$ is the line that contains $(B,p)$.
Then $q(u_i)-p(u_i)$ are parallel for $1\leq i\leq 4$.
Since $(B,p)$ and $(B,q)$ are equivalent, we have 
 $|| q(v_i) - q(u_i) || = || p(v_i) - p(u_i) ||$ for $i=1,2$.
But this gives 
$|| p(v_1) - p(u_1) || = || p(v_2) - p(u_2) ||$, contradicting the genericity of \(p\).

	\begin{figure}[ht]
		\centering
 \includegraphics{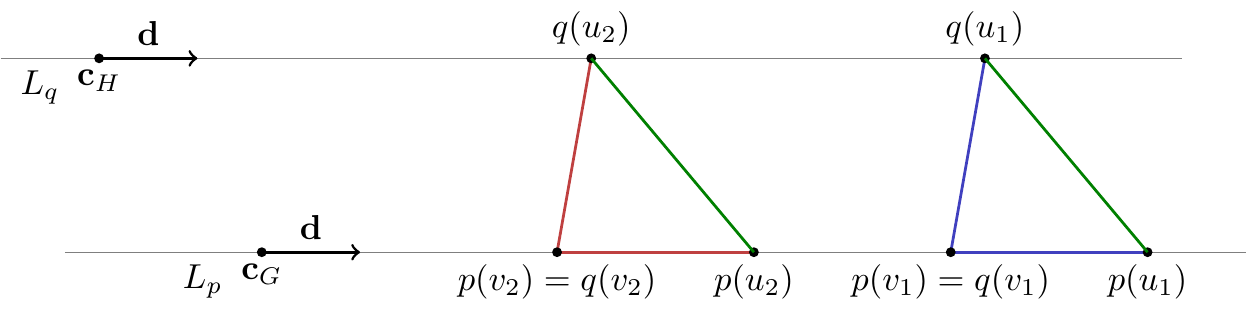}
 \caption{Diagram for the case \(\mathbf{d}_G = \mathbf{d}_H\) and \(\mathbf{c}_G \neq \mathbf{c}_H\). The green segments are parallel, the red segments have equal lengths, and the blue segments have equal lengths. It must be the case that the red and blue lengths are equal.}
 \label{fig:latter-case}
	\end{figure}

 Hence \(\mathbf{t} = \mathbf{0}\), which gives \(\mathbf{d}_G = \mathbf{d}_H\) and \(\mathbf{c}_G = \mathbf{c}_H\). Therefore \(p\) is congruent to \(q\), as desired.
	Thus \(B\) is generically universally rigid in $\R^1$.
\end{proof}

We can use Theorem~\ref{join} and the results of the previous section to construct an infinite family of
triangle free generically universally rigid graphs on the line. 

Theorem~\ref{join} does not hold for $k\leq 3$. 
For example, the join of two $K_3$ graphs along three edges (the so-called prism, or Desargue graph)
is not generically universally rigid in $\R^1$.
For a detailed analysis of this graph see \cite{C22}.

\subsection{Degree-2 extension}



Let $(G,p)$ be a framework on the line with $G=(V,E)$. 
A pair of vertices $\{u,v\}$, $u,v\in V$ 
is called \emph{universally linked} in $(G,p)$ 
if $\|q(u)-q(v)\|=\|p(u)-p(v)\|$ holds for all frameworks
$(G,q)$ which are equivalent to $(G,p)$ (in all dimensions).

We believe that if $\{u,v\}$ is not universally linked in $(H,p)$, and $(G,p)$ is
obtained from $(H,p)$ by a degree-2 extension, then $\{u,v\}$ is not
universally linked in $(G,p)$.
We can prove the following somewhat weaker statement.

\begin{thm}
\label{2add}
Suppose that $H$ is not generically universally rigid.
Let $G$ be obtained
from $H$ by a degree-2 extension. Then
$G$ is not generically universally rigid.
\end{thm}

\begin{proof}
We may assume that $|V(H)|\geq 4$, since the statement is easy to verify when $H$
has at most three vertices.
By our assumption $H$ is not generically universally rigid, hence there exists a 
generic 1-dimensional realization $(H,p)$ and a
pair
$\{u,v\}$ of vertices of $H$ which is not universally linked in $(H,p)$. Thus there exists an equivalent framework 
$(H,q)$, such that
$|p(u)-p(v)|\not=|q(u)-q(v)|$. Let $G$ be obtained from $H$ by adding a vertex $w$ and edges $wx,wy$.
Let $\alpha =|p(x)-p(y)|$ and $\beta =|q(x)-q(y)|$.
Since $p$ is generic, we have $\alpha \not= 0$. 

Suppose first that we have $\beta \geq \alpha$. Then we can choose a point $r$ on the line (in the complement of the
line segment connecting $p(x)$ and $p(y)$), 
for which the set $r\cup \{p(z) : z\in V(H)\}$ is generic, and $|r-p(x)|+|r-p(y)|\geq \beta$.
Since we also have $||r-p(x)|-|r-p(y)||=\alpha \leq \beta$, we can find a point $s$ (in a plane
that contains $q(x),q(y)$, and third point of $(H,q)$) for which $|s-q(x)|=|r-p(x)|$ and $|s-q(y)|=|r-p(y)|$ holds.
Then by adding $w$ to $(H,p)$ so that $p(w)=r$, and adding $w$ to $(H,q)$ so that $q(w)=s$ we obtain
a pair of equivalent realizations of $G$ for which $(G,p)$ is generic and
$\{u,v\}$ is not universally linked in $(G,p)$.

Next suppose that $\alpha > \beta$. If $\beta\not= 0$ then we can use a similar argument as follows.
We choose a point $r$ on the line (in the 
interior of the line segment connecting $p(x)$ and $p(y)$), 
for which the set $r\cup \{p(z) : z\in V(H)\}$ is generic, and $||r-p(x)|-|r-p(y)||\leq \beta$.
Since we also have $|r-p(x)|+|r-p(y)|=\alpha > \beta$, we can find a point $s$ (in a plane
that contains $q(x),q(y)$, and third point of $(H,q)$) for which $|s-q(x)|=|r-p(x)|$ and $|s-q(y)|=|r-p(y)|$ holds.
Then by adding $w$ to $(H,p)$ so that $p(w)=r$, and adding $w$ to $(H,q)$ so that $q(w)=s$ we obtain
a pair of equivalent realizations of $G$ for which $(G,p)$ is generic and
$\{u,v\}$ is not universally linked in $(G,p)$.
 
Finally, we consider the case when $\beta=0$, that is, $q(x)$ and $q(y)$ are coincident. Then we modify
$(H,q)$ by applying a result of Bezdek and Connelly \cite{BK}, which states that there is a continuous motion $\Phi:[0,1]\to \R^{2d|V|}$
from $(H,p)$ to $(H,q)$ in $\R^{2d}$, where $d$ is the affine dimension of $(H,q)$. Moreover, the distances between all
pairs of vertices change in a monotone way during the motion. 
If there is a realization $(H,q')$ on the trajectory of this motion for which
$q'(x)\not= q'(y)$ and
$|p(u)-p(v)|\not=|q'(u)-q'(v)|$,
then the theorem follows by applying the above arguments to $(H,q')$
in place of $(H,q)$.

If there is no such realization then
$x$ and $y$ becomes
coincident (at time $t\in [0,1]$, for some $0< t < 1$) before the distance between $u$ and $v$ begins to change.
However, there must be another pair of non-adjacent vertices $\{u',v'\}$ for which
the trajectory in $[0,t)$ contains a framework $(H,q')$ with
$q'(x)\not= q'(y)$ and
$|p(u')-p(v')|\not=|q'(u')-q'(v')|$. This follows from the fact that
the graph
$K_{n}-e$, obtained from a complete graph by deleting an edge, is generically universally rigid in $\R^1$, for $n\geq 4$ (which is easy to verify).
So it is not possible that all but one of the pairwise distances stay the same in, say $[0,\frac{t}{2}]$.
The theorem then follows by applying the above arguments to $(H,q')$
in place of $(H,q)$ and $\{u',v'\}$ in place of $\{u,v\}$.
\end{proof}

Note that Theorem~\ref{2add} does not hold if we replace degree-2 extension by degree-3 extension:
consider $K_4-e$ and remove a vertex of degree three.

\section{A lower bound on the number of edges}

The following question was posed in \cite{JN}.

\begin{question} \cite[Question 3.6]{JN}
\label{q}
Let $G=(V,E)$ be generically universally rigid in $\R^1$.
Does this imply that $|E|\geq 2|V|-3$?
\end{question}

If we replace universally rigid by globally rigid (resp. rigid), then the best possible
lower bound is $|V|$ (resp. $|V|-1$), which is attained by the family of cycles (resp.
trees). 
%
As an application of Theorem~\ref{2add} we show that 
generically universally rigid graphs need more edges indeed, namely, at least 
%
$\frac{3}{2}|V|$ (assuming $|V|\geq 6$). This bound is the first step towards an
affirmative answer to Question \ref{q}. Note that $2|V|-3$ would be best possible,
as shown by the graphs $K_{|V|-2,2}+e$, obtained from the complete bipartite graph $K_{|V|-2,2}$
by adding an edge.


\begin{thm}
\label{1.5}
Let $G=(V,E)$ be a graph with $|V|\geq 6$ and $|E| < \frac{3}{2}|V|$.
Then $G$ is not generically universally rigid in $\R^1$.
\end{thm}

\begin{proof}
By induction on $|V|$. For $|V|=6$ it is easy to see that no graph with at most
eight edges is generically universally rigid: suppose $H$ is such a graph. Then $H$ must have a 
vertex $v$ of degree two. By Theorem \ref{2add} $H-v$ is also generically universally rigid. It also has
a vertex $w$ of degree two. Then $H-\{v,w\}$ has four vertices, at most four edges,
and, again by Theorem~\ref{2add}, is generically universally rigid, which is impossible.

Consider the inductive step and let $|V|\geq 7$. Let us suppose, for a contradiction,
that $G$ is generically universally rigid.
The edge count and the
fact that generically universally rigid graphs on at least three vertices have minimum degree at least two
implies that $G$ has a vertex $v$ of degree two. By Theorem~\ref{2add} $G-v$ is generically universally rigid.
It implies, by induction, that $G-v$ has at least $\frac{3}{2}(|V|-1)$ edges.
This gives $|E|\geq |E(G-v)|+2 \geq \frac{3}{2}(|V|-1) + 2 \geq \frac{3}{2}|V|$,
a contradiction.
\end{proof}


\section{Concluding remarks}

Instead of the genericity assumption on $p$, we may consider frameworks
whose vertices are in general position, or frameworks for which $p$
is injective, or quasi-injective. (Quasi-injective means that the endvertices of an edge cannot
be coincident.) These properties can also be used to define families of graphs $G$
by requiring that some (or every) $1$-dimensional general position (resp. injective, quasi-injective) realization of $G$ is
universally rigid. The characterization of these families, in most cases, is still open.
See \cite[Table 3]{G} for a good overview.

To motivate further research in this direction, we show that the bound in Question \ref{q} above
is tight if we replace generic by quasi-injective.

\begin{thm}
\label{qi}
Suppose that every quasi-injective realization of $G=(V,E)$ in $\R^1$ is
universally rigid. Then $|E|\geq 2|V|-3$.
\end{thm}

\begin{proof}
Let us assume that $|E|\leq 2|V|-4$. By a result of Chen and Yu \cite{Chen}
this implies that $G$ has an independent vertex separator, that is, a set $S\subset V$ for
which $G-S$ is disconnected and $G$ has no edges with both endvertices in $S$.
We can use this fact to construct a quasi-injective $1$-dimensional realization $(G,p)$ which is
not universally rigid (in fact not even globally rigid): we choose a quasi-injective map $p$ so that all vertices
of $S$ are mapped to the same point $s$ on the line. Then an equivalent, non-congruent realization $(G,q)$ can be
obtained by rotating the vertices of some connected component of $G-S$ about $s$.
\end{proof}

Since every quasi-injective realization of the graphs $K_{|V|-2,2}+e$ on the line is universally rigid by Lemma~\ref{2ext},
the bound in Theorem~\ref{qi} is tight.

\section*{Acknowledgements}

The results of this paper were obtained in the framework of an undergraduate research
experience project of BSM (Budapest Semesters in Mathematics), led by the second author.
We thank Owen Cardwell and Keegan Stump
for some useful comments.

The second author was supported in part by the Hungarian Scientific
Research Fund grant no. K135421.

\end{document}